\documentclass[a4paper, 10pt]{article}
\usepackage[mathscr]{euscript}
\usepackage{amsthm}
\usepackage{amsmath}
\usepackage{amssymb}
\usepackage[english]{babel}
\usepackage{a4wide}
\usepackage{enumerate}

\title{On the Embedding of $BV$ Space into Besov-Orlicz Space}
\author{Aleksander Pawlewicz, Micha\l{} Wojciechowski}
\date{2021\\ October}

\makeatletter
\def\blfootnote{\gdef\@thefnmark{}\@footnotetext}
\makeatother

\theoremstyle{plain}
\newtheorem{theorem}{Theorem}
\newtheorem*{theorem*}{Theorem}

\newtheorem{lemma}[theorem]{Lemma}

\theoremstyle{remark}
\newtheorem{remark}{Remark}

\numberwithin{equation}{section}

\begin{document}

\maketitle

\begin{abstract}
We give a sufficient (and, in the case of a compact domain, a necessary) condition for the embedding of Sobolev space of functions with integrable gradient into Besov-Orlicz spaces to be bounded. The condition has  a form of a simple integral inequality involving Young and weight functions. We provide an example with Matuszewska-Orlicz indices of involved Orlicz norm equal to one. The main tool is the molecular decomposition of functions from a $BV$ space. 
\end{abstract}

\blfootnote{\textup{2020} \textit{Mathematics Subject Classification}: 46E30, 46E35, 30H25.}

\section*{Introduction}
The motivation to write this paper was the study of the classical Sobolev embedding theorems from the point of view of operator ideals. Factorizing the embedding $S:W^{1,1}(\mathbb{R}^d) \hookrightarrow L^p (\mathbb{R}^d)$ as
$A\circ B$ where $B:W^{1,1}(\mathbb{R}^d) \hookrightarrow X(\mathbb{R}^d)$ and 
$A:X(\mathbb{R}^d)\hookrightarrow L^p (\mathbb{R}^d)$, where $X(\mathbb{R}^d)$ is suitable chosen translationary invariant space of distributions is the main tool for this aim. The choice of the space $X$ is the key issue here - on one hand it should have an easy to handle norm and on the other  hand it should be as "close" to $W^{1,1}$ as possible. Then one can extract stronger operator ideal properties from the embedding $A:X(\mathbb{R}^d)\hookrightarrow L^p (\mathbb{R}^d)$ which is easier to study. Many classical Banach function spaces suitable for this role were studied. In the case of rearrangement invariant spaces, or Orlicz spaces the optimal results were obtained in \cite{Cia} and \cite{EdKePi}, respectively. The other known classical spaces $X$ which could be used in this role are the Besov spaces. 
Besov spaces were introduced to mathematics by Besov in \cite{Bes1} and \cite{Bes2} (see also \cite[page 34]{Tri}). Nowadays, there is an extensive literature devoted to these classical spaces, for example Peetre's book \cite{Pee} and several books written by Hans Triebel to mention only a few.
The boundedness of the embedding of Sobolev space into Besov space was proved in the 70s and 80s by Il’in (see book \cite[paragraph 18.12]{BeIlNi}) and in the anisotropic case by Kolyada \cite{Kol1}. For a brief description of these facts and some generalizations in the context of Lorentz spaces, we recommend Kolyada’s paper \cite{Kol2}.
Our starting point (especially concerning proofs) is the following vector valued version of Il'in's theorem.

\begin{theorem}[{\cite[Theorem 3.3, page 93]{PeWo}}]\label{Pelczynski}
Let $d=2, 3, ...$ and let $E$ be a Banach space. Then
$$W^{1,1}(\mathbb{R}^d,E)\hookrightarrow B_{p,1}^{\theta(p,d)}(\mathbb{R}^d,E),$$
where $\theta(p,d)=d(1/p+1/d-1)$ and $1<p<d/(d-1)$.
\end{theorem}
One has to stress that the Besov norm is much easier to handle than the original Sobolev one  - it could be expressed  both directly using the modulus of continuity or with help of its Fourier transform.

In this paper we will consider the general form of the Besov-Orlicz spaces studied in \cite{PiSi}. We find the Besov - Orlicz spaces which are "closer" to $W^{1,1}$ than any of the classical Besov spaces. Moreover, we give the necessary and sufficient conditions on functions
$\Phi$ and $\Psi$ for the boundedness of the embedding $W^{1,1}(\Omega)\hookrightarrow B^\Psi_{\Phi,1}(\Omega)$ in the case of compact domains $\Omega$.

In fact we are interested in $BV$ spaces rather than $W^{1,1}$. When the target space of the embedding is reflexive, the embedding of the Sobolev space can be immediately extended to the space of functions of bounded variation, since the latter one is a subspace of the bidual of $W^{1,1}$. However the Orlicz - Besov spaces $B^\Psi_{\Phi,1}(\Omega)$ appearing in our characterization may not be reflexive. Therefore to
perform our proof we need to deal directly with the space $BV$. In particular we adjust the "molecular decomposition" from \cite{PeWo} to the functions of bounded variation. Here the main obstacle is that smooth functions are not dense in the $BV$ space with respect to the $BV$ norm.

The main result of the paper is the following

\begin{theorem}\label{Kolyada_zwarte}
Let $B_{\Phi,1}^\Psi(\Omega)$ be a Besov-Orlicz space for some continuous function $\Psi$, a Young function $\Phi$ such that $\lim_{x\rightarrow\infty}\frac{\Phi(x)}{x^{d/(d-1)}}=0$ and a compact subset $\Omega$ of $\mathbb{R}^d$. Then $BV(\Omega)$ space can be continuously embedded into the space $B_{\Phi,1}^\Psi(\Omega)$ if and only if there exist a positive constant $D$ such that for every $s>0$ we have
\begin{eqnarray}
\label{Kolyada_nier_zw}
\frac{s^{d-1}}{\Phi^{-1}(s^d)}\int_0^s\frac{\Psi(1/t)}{t}dt + \int_s^\infty\frac{\Psi(1/t)s^{d-1}}{\Phi^{-1}(ts^{d-1})t} dt <D.
\end{eqnarray}
\end{theorem}

When a domain of considered functions is the whole space, $\Omega=\mathbb{R}^d$, then we do not have the equivalence of conditions.
\begin{theorem}
\label{Kolyada}
Let $B_{\Phi,1}^\Psi(\mathbb{R}^d)$ be a Besov-Orlicz space for some continuous function $\Psi$, and a Young function $\Phi$ such that $\lim_{x\rightarrow\infty}\frac{\Phi(x)}{x^{d/(d-1)}}=0$. If there exists a positive constant $D$ such that for every $s>0$ we have
\begin{eqnarray}
\label{Kolyada_nier}
\frac{s^{d-1}}{\Phi^{-1}(s^d)}\int_0^s\frac{\Psi(1/t)}{t}dt + \int_s^\infty\frac{\Psi(1/t)s^{d-1}}{\Phi^{-1}(ts^{d-1})t} dt <D,
\end{eqnarray}
then there exists a positive constant $\mathcal{C}$ such that
\begin{eqnarray}
\label{teza}
\lVert f\rVert_{B_{\Phi,1}^\Psi}
\leq
\mathcal{C}\lVert f\rVert_{BV},
\end{eqnarray}
for every $BV$ function $f$.
\end{theorem}

Similarly to the concept of paper \cite{PeWo}, we construct a molecular decomposition of the functions from $BV$ space. After that, using this decomposition and some facts from the theory of Orlicz spaces, we prove Theorem \ref{Kolyada_zwarte} and Theorem \ref{Kolyada}.

\bigskip




Let us write a few words about the organization of this paper. In Section 1, we define notions considered further and present basic relations between them. Section 2 is devoted to the molecular decomposition of $BV$ functions. In Section 3, we prove the main theorem of the paper in a case of functions defined on a fixed compact set (Theorem \ref{Kolyada_zwarte}). In fact, in this case we have an equivalence of conditions \eqref{Kolyada_nier} and \eqref{teza}. In Section 4, we prove Theorem \ref{Kolyada}.

\section{Notation and basic facts}
Let us introduce notation. The leter $d$ will always denote a positive integer. As usual, the $L^p$ norm of a (measurable) function $f$ will be denoted by $\lVert f\rVert_p$.
Let $\mu_1, ..., \mu_d$ be signed Radon measures on $\mathbb{R}^d$. A measure $\mu=\left(\mu_1, ..., \mu_d\right)$ is called a weak gradient of the function $f:\mathbb{R}^d\rightarrow\mathbb{R}$ if
$$\int_{\mathbb{R}^d}f(x)\mbox{div}\varphi (x)\,dx=
\int_{\mathbb{R}^d}f(x)\sum_{n=1}^d\frac{\partial\varphi_n}{\partial x_n}(x)\,dx=
-\int_{\mathbb{R}^d}\varphi (x)\cdot d\mu(x),$$
for every $C_0^\infty$ function $\varphi:\mathbb{R}^d\rightarrow\mathbb{R}^d$. The symbol $\nabla f$ will denote the measure $\mu$ and its total variation will be defined as
$$\lVert\nabla f\rVert_M=\sup\left\{\int_{\mathbb{R}^d}f\mbox{ div}\varphi\,dx: \varphi=\left(\varphi_1, ..., \varphi_d\right)\in C_0^\infty, |\varphi(x)|\leq 1\mbox{ for } x\in\mathbb{R}^d\right\}.$$

Sobolev space $W^{1,1}(\mathbb{R}^d)$ is the Banach space of (equivalence classes of) integrable functions $f:\mathbb{R}^d\rightarrow\mathbb{R}$ which have a weak gradient that is absolutely continuous with respect to the Lebesgue measure.
This space is equipped with the norm
$$
\lVert f\rVert_{1,1}
=
\lVert f\rVert_1+\lVert\nabla f\rVert_M 
=
\lVert f\rVert_1+\lVert\tilde{\nabla}f\rVert_1,
$$
where the symbol $\tilde{\nabla}f$ denote an ordinary gradient of the function $f$ which exists almost everywhere in this situation.

Analogously, $BV(\mathbb{R}^d)$ (space of functions of bounded variation) is the Banach space of (equivalence classes of) integrable functions $f:\mathbb{R}^d\rightarrow\mathbb{R}$ which have a weak gradient. The norm on this space is defined as follows
$$
\lVert f\rVert_{BV}
=
\lVert f\rVert_1+\lVert\nabla f\rVert_M.
$$

For more information about Sobolev and $BV$ spaces we recommend books \cite{EvGa} and \cite{Zie}.

\begin{remark}\label{uwaga1}
We have embeddings:
$$W^{1,1}(\mathbb{R}^d)\hookrightarrow BV(\mathbb{R}^d)\hookrightarrow L^\frac{d}{d-1}(\mathbb{R}^d).$$
\end{remark}

Remark \ref{uwaga1} is a consequence of the following theorem.
\begin{theorem}[Sobolev embedding,  {\cite[Theorem 5.10, page 216]{EvGa}}]\label{wlozenie}
For every $d=1, 2, ...$ there exists a positive constant $C$ such that for every $BV(\mathbb{R}^d)$ function $f:\mathbb{R}^d\rightarrow\mathbb{R}$ we have
$$\lVert f\rVert_\frac{d}{d-1}
\leq
C\lVert\nabla f\rVert_M.$$
\end{theorem}

The last space which will be important for our consideration is the Besov-Orlicz space. Let us introduce this concept.

By a \textit{Young function} we mean a continuous, strictly increasing, convex function $\Phi:[0,\infty)\rightarrow \mathbb{R}$ such that $\Phi(0)=0$, $\lim_{t\rightarrow\infty}t/\Phi(t)=0$ and $\lim_{t\rightarrow 0}\Phi(t)/t=0$.

In fact, we will mainly be interested in Young functions $\Phi$ which dominate near zero every $x^p$ function, $p>1$, meaning that 
$$\lim_{t\rightarrow 0}\frac{\Phi(t)}{t}=\lim_{t\rightarrow 0}\frac{t^p}{\Phi(t)}=0,$$
for every $p>1$.

For the Young function $\Phi$  and for some continuous non-negative function $\Psi$ we define the generalized Besov-Orlicz space:
$$B_{\Phi,1}^\Psi=\left\{f\in L_\Phi(\mathbb{R}^d):\int_0^\infty \Psi(t)\omega_\Phi(f,t)\,\frac{dt}{t}<\infty\right\},$$
where $L_\Phi(\mathbb{R}^d)$ is an ordinary Orlicz space of integrable functions on $\mathbb{R}^d$ with the Luxemburg norm
$$\lVert f\rVert_\Phi=\inf\left\{\lambda>0: \int_{\mathbb{R}^d}\Phi\left(\frac{|f(x)|}{\lambda}\right)\,dx\leq 1\right\}$$
and
\begin{eqnarray}\label{modulus}
\omega_\Phi(f,t)=\sup_{|h|\leq t}\lVert f(\cdot+h)-f(\cdot)\rVert_\Phi
\end{eqnarray}
being the integral modulus of continuity.
We define the norm on the space $B_{\Phi,1}^\Psi$ by
$$\lVert f\rVert_{B_{\Phi,1}^\Psi}=\lVert f\rVert_\Phi+\int_0^\infty \Psi(t)\omega_\Phi(f,t)\,\frac{dt}{t}.$$

\begin{remark}\label{uwaga2}
Besov-Orlicz spaces considered above are, in fact, a generalization of Besov spaces with the norm:
$$\lVert f\rVert_{B_{p,1}^\theta}=\lVert f\rVert_p+\int_0^\infty t^{-\theta}\omega_p(f,t)\,\frac{dt}{t}.$$
Thus, as one can see, $\Psi(t)=t^{-\theta}$ and $\Phi(t)=t^p$ in this case.
\end{remark}

We refer the Reader to the Introduction of papers \cite{PeWo} and \cite{Kol2} and the literature mentioned there for more information about the history of connections between Sobolev and Besov spaces.

The symbol $\chi_A$ will denote the characteristic function of the set $A\subseteq\mathbb{R}^d$, $d=1, 2, ...$, that is the function such that $\chi_A(x)=1$ if $x\in A$ and $\chi_A(x)=0$ if $x\not\in A$.

\section{Molecular decomposition of $BV$ functions}

In this section, we prove the molecular decomposition of $BV$ functions. An analogous decomposition for functions from Sobolev space was proved in the paper \cite[Theorem 2.1, pages 71 - 72]{PeWo}. The authors of that paper highly relied on the density of smooth functions in Sobolev space. Although we use the same key ingredients, our more general proof is slightly different and does not use the density of smooth functions.
\begin{theorem}\label{molecular}
For $d=1, 2, ...$ let $f$ be a $BV(\mathbb{R}^d)$ function with compact support. Then there exists a sequence $(f_n)_{n=0}^\infty$ of $BV(\mathbb{R}^d)$ functions (called molecules) and a constant $\alpha>0$ (which depends only on the dimention $d$) such that
\begin{equation}\label{r1}
f(x)=\sum_{n=0}^\infty f_n(x) \mbox{ for a.e. } x\in\mathbb{R}^d,
\end{equation}
\begin{equation}\label{r2}
\lVert f\rVert_1=\sum_{n=0}^\infty \lVert f_n\rVert_1 \mbox{   and   } \lVert\nabla f\rVert_M=\sum_{n=0}^\infty \lVert\nabla f_n\rVert_M,
\end{equation}
and
\begin{equation}\label{r3}
\lVert f_n\rVert_\infty^{1/d}\lVert f_n\rVert_1^{(d-1)/d}
\leq
\alpha\lVert\nabla f_n\rVert_M \mbox{ for } n=0, 1, ...\, .
\end{equation}
\end{theorem}
\begin{proof}
At first, let us assume that $f(x)\geq 0$ for every $x\in\mathbb{R}^d$. We will define inductively a sequence of positive real numbers $(a_n)_{n=0}^\infty$ and a family $(A_n)_{n=0}^\infty$ of subsets of the space $\mathbb{R}^d$. 
Put
$$a_0=0$$
and 
$$A_0=\{x\in\mathbb{R}^d: f(x)>a_0\}.$$
Notice that $|A_0|$, the measure of the set $A_0$, is finite because $f$ has compact support. Now assume that we have already defined elements $a_n$ and $A_n$. Then put
$$a_{n+1}=\inf\left\{t>a_n: |\{f(x)>t\}|\leq\frac{1}{2}|A_n|\right\}$$
and 
$$A_{n+1}=\{x\in\mathbb{R}^d: f(x)>a_{n+1}\}.$$
Without loss of generality we can assume that $|A_n|>0$ for every $n=0, 1, ...\, .$

Now we define the functions $f_n$, for $n=0, 1, ...\, .$ Put
\begin{equation}
 f_n(x) = \left.
  \begin{cases}
    0, & \text{for } f(x)\leq a_n \\
    f(x)-a_n, & \text{for } a_n < f(x) \leq a_{n+1} \\
    a_{n+1}-a_n, & \text{for } a_{n+1} < f(x)
  \end{cases}
  \right. ,
\end{equation}
for $n=0, 1, ...$ .
Notice that $\left\{x\in\mathbb{R}^d: f_n>0\right\}=A_n$.
Then for $n=0, 1, ...$ we have
\begin{eqnarray*}
\lVert\nabla f_n\rVert_M
=
\int_0^\infty\lVert\nabla\chi_{\{f_n(x)>t\}}\rVert_M\, dt 
=
\int_0^{a_{n+1}-a_n}\lVert\nabla\chi_{\{f_n(x)>t\}}\rVert_M\, dt
\end{eqnarray*}
by coarea formula for $BV$ functions \cite[Theorem 5.4.4, Chapter 5, page 231]{Zie}.
Now we will use isoperimetric inequality \cite[Theorem 5.4.3, Chapter 5, page 230]{Zie} to get
\begin{eqnarray*}
\int_0^{a_{n+1}-a_n}\lVert\nabla\chi_{\{f_n(x)>t\}}\rVert_M\, dt
&\geq&
\frac{1}{C}\int_0^{a_{n+1}-a_n}|\{f_n(x)>t\}|^\frac{d-1}{d}\,dt \\
&=&
\frac{1}{C}\int_{a_n}^{a_{n+1}}|\{f(x)>t\}|^\frac{d-1}{d}\,dt \\
&\geq&
\frac{1}{C} \frac{a_{n+1}-a_n}{2} \left|\left\{f(x)>\frac{a_{n+1}+a_n}{2}\right\}\right|^\frac{d-1}{d} \\
&\geq&
\frac{1}{2C} \lVert f_n\rVert_\infty \left(\frac{1}{2}|\{f(x)>a_n\}|\right)^\frac{d-1}{d} \\
&=&
\frac{1}{2C} \lVert f_n\rVert_\infty \left(\frac{1}{2}|A_n|\right)^\frac{d-1}{d} \\
&\geq&
\frac{1}{2^{2-\frac{1}{d}}C} \lVert f_n\rVert_\infty^{1/d} \lVert f_n\rVert_1^\frac{d-1}{d}. 
\end{eqnarray*}
Putting the above calculations together, we get
\begin{equation}
\lVert f_n\rVert_\infty^{1/d} \lVert f_n\rVert_1^\frac{d-1}{d} \leq 2^{2-\frac{1}{d}}C\lVert\nabla f_n\rVert_M,
\end{equation}
for $n=0, 1, ...$ . This proves \eqref{r3}. Properties \eqref{r1} and \eqref{r2} are true because
\begin{eqnarray*}
\lVert f\rVert_1
&=&
\int_0^{\lVert f\rVert_\infty} |\{f(x)>t\}|\, dt \\
&=&
\sum_{n=0}^\infty \int_{a_n}^{a_{n+1}} |\{f(x)>t\}|\, dt \\
&=&
\sum_{n=0}^\infty \int_0^{a_{n+1}-a_n} |\{f(x)>t+a_n\}|\, dt \\
&=&
\sum_{n=0}^\infty \int_0^{a_{n+1}-a_n} |\{f_n(x)>t\}|\, dt \\
&=&
\sum_{n=0}^\infty \lVert f_n\rVert_1
\end{eqnarray*}
and, by coarea formula for $BV$ functions, we have
\begin{eqnarray*}
\lVert \nabla f\rVert_M
&=&
\int_0^{\lVert f\rVert_\infty} \lVert\nabla\chi_{\{f(x)>t\}}\rVert_M\, dt \\
&=&
\sum_{n=0}^\infty \int_{a_n}^{a_{n+1}} \lVert\nabla\chi_{\{f(x)>t\}}\rVert_M\, dt \\
&=&
\sum_{n=0}^\infty \int_0^{a_{n+1}-a_n} \lVert\nabla\chi_{\{f(x)>t+a_n\}}\rVert_M\, dt \\
&=&
\sum_{n=0}^\infty \int_0^{a_{n+1}-a_n} \lVert\nabla\chi_{\{f_n(x)>t\}}\rVert_M\, dt \\
&=&
\sum_{n=0}^\infty \lVert \nabla f_n\rVert_M.
\end{eqnarray*}
Now assume that $f$ is a $BV(\mathbb{R}^d)$ function with real values. Then we can write $f$ as a sum of two non-negative functions:
$$f=f^+-f^-,$$
where
\begin{equation*}
 f^+(x) = \left.
  \begin{cases}
   0, & \text{for } f(x)\leq 0 \\
    f(x) & \text{for } f(x)>0
  \end{cases}
  \right. 
\end{equation*}
and
$$f^-(x)=f^+(x)-f(x),$$
for $x\in\mathbb{R}^d$. Notice that $\{f^+(x)>0\}\cap\{f^-(x)>0\}=\emptyset$. For both functions $f^+$ and $f^-$ we can get a molecular decomposition as in the first part of the proof:
$$f^+=\sum_{n=0}^\infty f_n^+ \mbox{ and } f^-=\sum_{n=0}^\infty f_n^- \mbox{ for a. e. } x\in\mathbb{R}^d.$$
Then the sequence
$$\{f_0^+, -f_0^-, f_1^+, -f_1^-, ... \}$$
is the desired molecular decomposition.
\end{proof}

\section{Proof of Theorem \ref{Kolyada_zwarte} (the case of a compact domain)}

In this section we will prove a generalization of embedding BV spaces into Besov spaces in the case when all considered functions are equal $0$ outside a fixed, compact subset of $\mathbb{R}^d$. Thus, in this section we will consider $BV(\Omega)$ and $B_{\Phi,1}^\Psi(\Omega)$, for some fixed, compact set $\Omega\subseteq\mathbb{R}^d$. Actually, we will prove Theorem \ref{Kolyada_zwarte}.

Before we present the proof of this theorem we need the following fact about the integral modulus of continuity (for definition see \eqref{modulus}).
\begin{lemma}\label{omega1}
Let $d=1, 2, ...$ . For $f\in BV(\mathbb{R}^d)$ and $t>0$ we have
$$\omega_1(f,t)\leq t\lVert\nabla f\rVert_M.$$
\end{lemma}
\begin{proof}
We will use the following notation: $(x,h)_i=(x_1-h_1, ..., x_i-h_i, x_{i+1}, ..., x_d),$
for $x$ and $h$ in $\mathbb{R}^d$ and $i\in\{1, 2, ..., d\}$;
$[s,x]_i=(x_1, ..., x_{i-1}, s, x_{i+1}, ..., x_d),$
for $s\in\mathbb{R}$, $x\in\mathbb{R}^d$ and $i\in\{1, 2, ..., d\}$.
By Riesz representation theorem we have
\begin{eqnarray*}
\lVert f(\cdot+h)-f(\cdot)\rVert_1
&=&
\left\lVert\left(f(\cdot+h)-f(\cdot)\right)dx\right\rVert_M \\
&=&
\sup_{\substack{g\in C_0 \\ \lVert g\rVert_\infty\leq 1}}
\int_{\mathbb{R}^d}\left[f(x+h)-f(x)\right]g(x)\,dx \\
&=&
\sup_{\substack{g\in C_0 \\ \lVert g\rVert_\infty\leq 1}}
\int_{\mathbb{R}^d}f(x)\left[g(x-h)-g(x)\right]\,dx \\
&=&
\sup_{\substack{g\in C_0 \\ \lVert g\rVert_\infty\leq 1}}
\int_{\mathbb{R}^d}f(x)\left\{\sum_{k=1}^d
\left[g((x,h)_k)-g((x,h)_{k-1})\right]
\right\}\,dx \\
&=&
\sup_{\substack{g\in C_0 \\ \lVert g\rVert_\infty\leq 1}}
\int_{\mathbb{R}^d}f(x)\left\{
\sum_{k=1}^d\frac{\partial}{\partial x_k}
\int_0^{h_k}
g\left([s,(x,h)_{k-1}]_k\right)
\,ds
\right\}\,dx \\
&=&
\sup_{\substack{g\in C_0 \\ \lVert g\rVert_\infty\leq 1}}
\int_{\mathbb{R}^d}\left(
\int_0^{h_k}
g\left([s,(x,h)_{k-1}]_k\right)
\,ds
\right)_{k=1}^d
\cdot\,d(\nabla f)(x) \\
&\leq&
|h|\lVert\nabla f\rVert_M,
\end{eqnarray*}
where by 
$\left(\int_0^{h_k}g\left([s,(x,h)_{k-1}]_k\right)\,ds\right)_{k=1}^d$
we denote the $d$-dimensional vector.
Taking the supremum of both sides over $|h|<t$ we get
$$\omega_1(f,t)\leq t\lVert\nabla f\rVert_M.$$
\end{proof}

One more ingredient for the proof of Theorem \ref{Kolyada_zwarte} is needed, a geometrical lemma. We will use it to prove necessity of our condition. In this section we will denote by $\lambda_d$ the $d$-dimensional Lebesgue measure.
\begin{lemma}
\label{geometrical}
Let $\mathcal{B}_d(0,r)$ be a closed $d$-dimensional ball centred at zero and of radius $r>0$. Moreover let $V_d$ denote the volume of $\mathcal{B}_d(0,1)$,
$$V_d=\lambda_d\big(\mathcal{B}_d(0,1)\big)$$
and let $\alpha$ be a real number such that $0\leq\alpha<r$. Then
$$\lambda_d\Big(\big(\mathcal{B}_d(0,r)\cup\mathcal{B}_d(x,r)\big)\setminus\big(\mathcal{B}_d(0,r)\cap\mathcal{B}_d(x,r)\big)\Big)\geq V_dr^{d-1}\alpha,$$
where $x$ is a point of $\mathbb{R}^d$ such that $|x|=2\alpha$.
\end{lemma}
We leave the proof of Lemma \ref{geometrical} for the Reader.
Now comes the time for the proof of the main theorem of this section.

\begin{proof}[Proof of Theorem \ref{Kolyada_zwarte}]
At first we will estimate the value 
$$\int_0^\infty\Psi(t)\omega_\Phi(f,t)\,\frac{dt}{t},$$
for $f\in BV(\Omega)$.
In order to do that, let us use the molecular decomposition. Let
$$f=\sum_{m=1}^\infty f_m,$$
as in point \eqref{r1} of Theorem \ref{molecular}. We have
\begin{eqnarray*}
\int_0^\infty\Psi(t)\omega_\Phi(f,t)\,\frac{dt}{t}
=
\int_0^\infty\Psi(1/t)\omega_\Phi(f,1/t)\,\frac{dt}{t}
&=&
\int_0^\infty\Psi(1/t)\omega_\Phi\Big(\sum_{m=1}^\infty f_m,1/t\Big)\,\frac{dt}{t}\\
&\leq&
\sum_{m=1}^\infty\int_0^\infty\Psi(1/t)\omega_\Phi(f_m,1/t\Big)\,\frac{dt}{t}.
\end{eqnarray*}
Now we will estimate each summand separately. So let $f_m$ be a molecule.

At the beginning we make a simple observation. The set
$$\Big\{\lambda>0: \int_{\mathbb{R}^d}\Phi\Big(\frac{|f_m(x+h)-f_m(x)|}{\lambda}\Big)\,dx\leq 1\Big\}$$ 
contains the set
$$\Big\{\lambda>0: \int_{\mathbb{R}^d}\frac{|f_m(x+h)-f_m(x)|}{2||f_m||_{L_\infty}}\Phi\Big(\frac{2||f_m||_{L_\infty}}{\lambda}\Big)\,dx\leq 1\Big\},$$
because the inequality $\Phi(\alpha a)\leq\alpha\Phi(a)$ is true for all $a>0$, $\alpha\in[0,1]$ and a convex function $\Phi$ such that $\Phi(0)=0$.
Computing the infima of the above sets we get
\begin{eqnarray}\label{infima}
||f_m(\cdot+h)-f_m(\cdot)||_{L_\Phi}\leq\frac{2||f_m||_{L_\infty}}{\Phi^{-1}\Big(\frac{2||f_m||_{L_\infty}}{||f_m(\cdot+h)-f_m(\cdot)||_{L_1}}\Big)}.
\end{eqnarray}

Now, if we take supremum over $|h|<1/t$ we get
$$\omega_\Phi(f_m,1/t)\leq\frac{2||f_m||_{L_\infty}}{\Phi^{-1}\Big(\frac{2||f_m||_{L_\infty}}{\omega_{L_1}(f_m,1/t)}\Big)}.$$
By Lemma \ref{omega1} we have
$$\omega_{L_1}(f_m,1/t)\leq \frac{||\nabla f_m||_M}{t},$$
therefore
\begin{eqnarray}
\label{big}
\omega_\Phi(f_m,1/t)
&\leq&
\frac{2||f_m||_{L_\infty}}{\Phi^{-1}\Big(\frac{2t||f_m||_{L_\infty}}{||\nabla f_m||_M}\Big)}.
\end{eqnarray}
We will use this estimate for "big" values of $t$. 

For "small" values of $t$ we need other estimate. Let us assume for the moment that the dimension $d$ is at least $2$. By \eqref{infima} and the inequality $||f_m(\cdot+h)-f_m(\cdot)||_{L_1}\leq 2||f_m||_{L_1}$ we get
$$||f_m(\cdot+h)-f_m(\cdot)||_{L_\Phi}\leq\frac{2||f_m||_{L_\infty}}{\Phi^{-1}\Big(\frac{||f_m||_{L_\infty}}{||f_m||_{L_1}}\Big)}.$$

By \eqref{r3} we get
$$\frac{||f_m||_{L_\infty}}{||f_m||_{L_1}}=\frac{||f_m||_{L_\infty}^{d/(d-1)}}{||f_m||_{L_1}||f_m||_{L_\infty}^{1/(d-1)}}\geq\frac{1}{\alpha}\Big(\frac{||f_m||_{L_\infty}}{||\nabla f_m||_M}\Big)^{d/(d-1)},$$
for absolute constant $\alpha\geq 1$. Hence
\begin{eqnarray}
\label{small}
||f_m(\cdot+h)-f_m(\cdot)||_{L_\Phi}
\leq
\frac{2||f_m||_{L_\infty}}{\Phi^{-1}\Big(\frac{1}{\alpha}\Big(\frac{||f_m||_{L_\infty}}{||\nabla f_m||_M}\Big)^{d/(d-1)}\Big)}\leq
2^\frac{d}{d-1}\alpha\frac{2||f_m||_{L_\infty}}{\Phi^{-1}\Big(\Big(\frac{2||f_m||_{L_\infty}}{||\nabla f_m||_M}\Big)^{d/(d-1)}\Big)}.
\end{eqnarray}

Now we define
$$s_m=\Big(\frac{2||f_m||_{L_\infty}}{||\nabla f_m||_M}\Big)^{1/(d-1)}$$
and using this value $s_m$ we divide the set of integration into two parts: 
\begin{eqnarray*}
\int_0^\infty\Psi(1/t)\omega_\Phi(f_m,1/t\Big)\,\frac{dt}{t}
&=&
\int_0^{s_m}\Psi(1/t)\omega_\Phi(f_m,1/t\Big)\,\frac{dt}{t}
+
\int_{s_m}^\infty\Psi(t)\omega_\Phi(f_m,1/t\Big)\,\frac{dt}{t} \\
&=&
I+II.
\end{eqnarray*}
For both integrals $I$ and $II$ we will need different estimates of the integrands.
For integral $I$ we get by \eqref{small}
\begin{eqnarray*}
I
&=&
\int_0^{s_m}\Psi(1/t)\omega_\Phi(f_m,1/t\Big)\,\frac{dt}{t} \\
&\leq&
2^\frac{d}{d-1}\alpha\int_0^{s_m}\Psi(1/t)\frac{2||f_m||_{L_\infty}}{\Phi^{-1}\Big(\Big(\frac{2||f_m||_{L_\infty}}{||\nabla f_m||_M}\Big)^{d/(d-1)}\Big)}\,\frac{dt}{t} \\
&=&
2^\frac{d}{d-1}\alpha\int_0^{s_m}\Psi(1/t)\frac{s_m^{d-1}}{\Phi^{-1}\Big(s_m^d\Big)}\,\frac{dt}{t} ||\nabla f_m||_M.
\end{eqnarray*}

For integral $II$, by \eqref{big}, we have
\begin{eqnarray*}
II
&=&
\int_{s_m}^\infty\Psi(1/t)\omega_\Phi(f_m,1/t\Big)\,\frac{dt}{t} \\
&\leq&
\int_{s_m}^\infty\Psi(1/t)\frac{2||f_m||_{L_\infty}}{\Phi^{-1}\Big(\frac{2t||f_m||_{L_\infty}}{||\nabla f_m||_M}\Big)}\,\frac{dt}{t} \\
&=&
\int_{s_m}^\infty\Psi(1/t)\frac{s_m^{d-1}}{\Phi^{-1}\left(ts_m^{d-1}\right)}\,\frac{dt}{t}||\nabla f_m||_M.
\end{eqnarray*}

We can sum up the above estimates, and using estimate \eqref{Kolyada_nier_zw} write
\begin{eqnarray*}
\int_0^\infty\Psi(t)\omega_\Phi(f,t)\,\frac{dt}{t}
&=&
\int_0^\infty\Psi(1/t)\omega_\Phi(f,1/t)\,\frac{dt}{t} \\
&\leq&
\sum_{m=1}^\infty\int_0^\infty\Psi(1/t)\omega_\Phi(f_m,1/t\Big)\,\frac{dt}{t} \\
&=&
\sum_{m=1}^\infty\left[
\int_0^{s_m}\Psi(1/t)\omega_\Phi(f_m,1/t\Big)\,\frac{dt}{t}
+
\int_{s_m}^\infty\Psi(1/t)\omega_\Phi(f_m,1/t\Big)\,\frac{dt}{t}
\right] \\
&\leq&
\sum_{m=1}^\infty\left[
2^\frac{d}{d-1}\alpha\int_0^{s_m}\frac{\Psi(1/t)s_m^{d-1}}{\Phi^{-1}\Big(s_m^d\Big)}\,\frac{dt}{t}
+
\int_{s_m}^\infty\frac{\Psi(1/t)s_m^{d-1}}{\Phi^{-1}\left(ts_m^{d-1}\right)}\,\frac{dt}{t}
\right] ||\nabla f_m||_M \\
&\leq&
2^\frac{d}{d-1}\alpha D
\sum_{m=1}^\infty||\nabla f_m||_M  \\
&=&
2^\frac{d}{d-1}\alpha D||\nabla f||_M 
\end{eqnarray*}

In dimension $d=1$ we know from Theorem \ref{molecular}, point \eqref{r3} that there exists a constant $\alpha\geq 1$ such that for every $m=1, 2, ...$ we have
$$||f_m||_\infty\leq \alpha||\nabla f_m||_M.$$
Thus, by \eqref{small} and above inequality we have
\begin{eqnarray*}
\int_0^\infty\Psi(t)\omega_\Phi(f_m,t)\,\frac{dt}{t}
&\leq&
\int_0^\infty\Psi(t)\frac{2||f_m||_{L_\infty}}{\Phi^{-1}\Big(\frac{2||f_m||_{L_\infty}}{t||\nabla f_m||_M}\Big)}\,\frac{dt}{t} \\
&\leq&
\int_0^\infty\Psi(t)\frac{2||f_m||_{L_\infty}}{\Phi^{-1}\Big(\frac{2||f_m||_{L_\infty}}{2\alpha t||\nabla f_m||_M}\Big)}\,\frac{dt}{t} \\
&\leq&
\int_0^\infty\Psi(t)\frac{2||f_m||_{L_\infty}}{\Phi^{-1}(1/t)\frac{||f_m||_{L_\infty}}{\alpha||\nabla f_m||_M}}\,\frac{dt}{t} \\
&=&
2\alpha\int_0^\infty\Psi(t)\frac{1}{\Phi^{-1}(1/t)}\,\frac{dt}{t} ||\nabla f_m||_M \\
&=&
2\alpha\left[\int_0^1\Psi(t)\frac{1}{\Phi^{-1}(1/t)}\,\frac{dt}{t}
+
\int_1^\infty\Psi(t)\frac{1}{\Phi^{-1}(1/t)}\,\frac{dt}{t}
\right] ||\nabla f_m||_M \\
&\leq&
2\alpha\left[\frac{1}{\Phi^{-1}(1)}\int_0^1\Psi(t)\,\frac{dt}{t}
+
\int_1^\infty\Psi(t)\frac{1}{\Phi^{-1}(1/t)}\,\frac{dt}{t}
\right] ||\nabla f_m||_M \\
&\leq&
2\alpha D||\nabla f_m||_M,
\end{eqnarray*}
by \eqref{Kolyada_nier_zw} for $s=1$. 

So we get 
\begin{equation}
\label{estimate1}
\int_0^\infty\Psi(t)\omega_\Phi(f_m,t)\,\frac{dt}{t}\leq 4\alpha D||\nabla f_m||_M,
\end{equation}
for every $BV(\Omega)$ function $f$.

To complete this part of the proof, we only need to show the existence of a positive constant $A=A(\Phi,d)$ such that
\begin{equation}
\label{estimate2}
\lVert f\rVert_\Phi
\leq
A\left(\lVert f\rVert_1 + \lVert\nabla f\rVert_M\right).
\end{equation}
In order to prove that, notice that there exist $N>0$ such that for all $x\geq N$ we have
$$\Phi(x)\leq x^\frac{d}{d-1}.$$
This is the consequence of $\lim_{x\rightarrow\infty}\frac{\Phi(x)}{x^{d/(d-1)}}=0$.

Thus, by Minkowski inequality, inequality $\Phi(\alpha t)\leq\alpha\Phi(t)$, for $0\leq\alpha\leq1$, $t\geq 0$, the above estimate, and Theorem \ref{wlozenie} we have
\begin{eqnarray*}
\lVert f\rVert_\Phi
&=&
\inf\left\{\lambda>0: \int_{\mathbb{R}^d}\Phi\left(\frac{|f(x)|}{\lambda}\right)\,dx\leq 1\right\} \\
&\leq&
\inf\left\{\lambda>0: \int_{\mathbb{R}^d\cap\left\{x :\frac{|f(x)|}{\lambda}\leq1\right\}}\Phi\left(\frac{|f(x)|}{\lambda}\right)\,dx\leq 1\right\} \\
&+&
\inf\left\{\lambda>0: \int_{\mathbb{R}^d\cap\left\{x :\frac{|f(x)|}{\lambda}>1\right\}}\Phi\left(\frac{|f(x)|}{\lambda}\right)\,dx\leq 1\right\} \\
&\leq&
\Phi(1)\inf\left\{\lambda>0: \int_{\mathbb{R}^d\cap\left\{x :\frac{|f(x)|}{\lambda}\leq1\right\}} \frac{|f(x)|}{\lambda}\,dx\leq 1\right\} \\
&+&
N\inf\left\{\lambda>0: \int_{\mathbb{R}^d\cap\left\{x :\frac{|f(x)|}{\lambda}>1\right\}} \left(\frac{|f(x)|}{\lambda}\right)^{d/(d-1)}\,dx\leq 1\right\} \\
&\leq&
\Phi(1)\lVert f\rVert_1 + N\lVert f\rVert_{d/(d-1)} \\
&\leq&
\max\left\{\Phi(1),N\right\}
C\left(\lVert f\rVert_1 + \lVert\nabla f\rVert_M\right).
\end{eqnarray*}
This ends the proof of \eqref{estimate2}.

As a consequence of the above calculations, in particular \eqref{estimate1} and \eqref{estimate2}, we have
\begin{eqnarray*}
\lVert f\rVert_{B_{\Phi,1}^\Psi}
&\leq&
\lVert f\rVert_\Phi + 
\int_0^\infty\Psi(t)\omega_\Phi(f,t)\,\frac{dt}{t} \\
&\leq&
\lVert f\rVert_\Phi + 
\sum_{m=1}^\infty\int_0^\infty\Psi(t)\omega_\Phi(f_m,t)\,\frac{dt}{t} \\
&\leq&
A\left(\lVert f\rVert_1 + \lVert\nabla f\rVert_M\right)
+
2CD\sum_{m=1}^\infty||\nabla f_m||_M \\
&=&
A\left(\lVert f\rVert_1 + \lVert\nabla f\rVert_M\right)
+
2CD||\nabla f||_M \\
&\leq&
\mathcal{C}\left(\lVert f\rVert_1 + \lVert\nabla f\rVert_M\right),
\end{eqnarray*}
by \eqref{r2}. This ends the proof of the sufficiency of condition \eqref{Kolyada_nier_zw}.

The necessity of condition \eqref{Kolyada_nier_zw} will be proved by contradiction. Let us assume that there exists a constant $C>0$ such that 
\begin{eqnarray}\label{eq7}
||f||_{B_{\Phi,1}^\Psi}\leq C||f||_{BV}
\end{eqnarray}
for every $f\in BV(\Omega)$ and that for every $D>0$ there exist $s\geq 0$ such that 
$$\frac{s^{d-1}}{\Phi^{-1}(s^d)}\int_0^s\frac{\Psi(1/t)}{t}dt + \int_s^\infty\frac{\Psi(1/t)s^{d-1}}{\Phi^{-1}(ts^{d-1})t} dt > D.$$
Define a function $f\in BV(\Omega)$ by putting $f(x)=1$ for $x\in\mathcal{B}_d(0,r)$ and $f(x)=0$ for $x\not\in\mathcal{B}_d(0,r)$, where $\mathcal{B}_d(0,r)$ means the $d$-dimensional ball centered at $0$ and of radius $r>0$.

Now we can compute appropriate norms. We have
\begin{eqnarray*}
||f||_{BV}
&=&
||f||_{L_1}+||\nabla f||_M \\
&=&
\lambda_d\left(\mathcal{B}_d(0,r)\right)+\lambda_{d-1}\left(\partial \mathcal{B}_d(0,r)\right) \\
&=&
V_dr^d+dV_dr^{d-1} \\
&\leq&
V_d\left[r^d+dr^{d-1}\right] \\
&\leq&
\left(\mbox{diam}(\Omega) + d\right)V_dr^{d-1},
\end{eqnarray*}
where $\mbox{diam}(\Omega)$ means the diameter of the set $\Omega$. By Lemma \ref{geometrical} ($\chi_A$ denotes the characteristic function of the set $A$) we get
\begin{eqnarray*}
\int_0^\infty\Psi(t)\omega_\Phi(f,t)\,\frac{dt}{t}
&=&
\int_0^\infty\Psi(1/t)\omega_\Phi(\chi_{\mathcal{B}_d(0,r)},1/t)\,\frac{dt}{t} \\
&=&
\int_0^\frac{1}{2r} \Psi\left(\frac{1}{t}\right)\omega_\Phi\left(\chi_{\mathcal{B}_d(0,r)},\frac{1}{t}\right)\,\frac{dt}{t}
+
\int_\frac{1}{2r}^\infty\Psi\left(\frac{1}{t}\right)\omega_\Phi\left(\chi_{\mathcal{B}_d(0,r)},\frac{1}{t}\right)\,\frac{dt}{t} \\
&\geq&
\int_0^\frac{1}{2r} \Psi\left(\frac{1}{t}\right)\frac{1}{\Phi^{-1}\left(\frac{1}{2V_d r^d}\right)}\,\frac{dt}{t}
+
\int_\frac{1}{2r}^\infty\Psi\left(\frac{1}{t}\right)\frac{1}{\Phi^{-1}\Big(\frac{t}{V_d r^{d-1}}\Big)}\,\frac{dt}{t} \\
&\geq&
\int_0^\frac{1}{2r} \Psi\left(\frac{1}{t}\right)\frac{1}{\Phi^{-1}\left(\frac{2^d}{V_d (2r)^d}\right)}\,\frac{dt}{t}
+
\int_\frac{1}{2r}^\infty \Psi\left(\frac{1}{t}\right)\frac{1}{\Phi^{-1}\Big(\frac{2^d  t}{V_d (2r)^{d-1}}\Big)}\,\frac{dt}{t} \\
&\geq&
\int_0^\frac{1}{2r} \Psi\left(\frac{1}{t}\right)\frac{1}{\frac{2^d}{V_d}\Phi^{-1}\left(\frac{1}{(2r)^d}\right)}\,\frac{dt}{t}
+
\int_\frac{1}{2r}^\infty \Psi\left(\frac{1}{t}\right)\frac{1}{\frac{2^d}{V_d}\Phi^{-1}\left(\frac{t}{(2r)^{d-1}}\right)}\,\frac{dt}{t} \\
&=&
\frac{V_dr^{d-1}}{2}
\left\{\frac{\frac{1}{(2r)^{d-1}}}{\Phi^{-1}\left(\frac{1}{(2r)^d}\right)}\int_0^\frac{1}{2r}\frac{\Psi(1/t)}{t}\,dt
+
\int_\frac{1}{2r}^\infty\frac{\Psi(1/t)\frac{1}{(2r)^{d-1}}}{t\Phi^{-1}\Big(\frac{t}{(2r)^{d-1}}\Big)}\,dt\right\} \\
&\geq&
\frac{1}{2\left(\mbox{diam}(\Omega) + d\right)}D||f||_{BV}.
\end{eqnarray*}
In the above computations we used concavity of the function $\Phi^{-1}$. This and $\Phi^{-1}(0)=0$ allow us to write $\Phi^{-1}(\alpha x)\leq\alpha\Phi^{-1}(x)$ for $x\geq 0$ and $\alpha\geq 1$. Notice that 
$$\frac{2^d}{V_d}\geq 1.$$

Now if we select $r$ such that $D>4\left(\mbox{diam}(\Omega) + d\right)C$, then
\begin{eqnarray*}
\lVert f\rVert_{B_{\Phi,1}^\Psi}
\geq
||f||_{L_\Phi}+\int_0^\infty\frac{\Psi(t)}{t}\omega_\Phi(f,t)\,dt 
\geq
2C\lVert f\rVert_{BV}
\end{eqnarray*}
which contradicts \eqref{eq7}.
\end{proof}

\begin{remark}\label{uwaga3}
Notice that condition \eqref{Kolyada_nier_zw} of Theorem \ref{Kolyada_zwarte} implies the existence of a constant $\mathcal{C}>0$, which depends only on $\Phi$ and $d$, such that
$$\lVert f\rVert_{B_{\Phi,1}^\Psi}\leq \mathcal{C}\lVert f\rVert_{BV},$$
for every $f\in BV(\Omega)$.
\end{remark}

\section{Proof of Theorem \ref{Kolyada}}
Let $f\in BV(\mathbb{R}^d)$. At first, let us assume that $f(x)\geq 0$ for every $x\in\mathbb{R}^d$. Consider a sequence of functions $(f_m)_{m=1}^\infty$ given by
\begin{equation}\label{pomocniczy}
 f_m(x) = \left.
  \begin{cases}
    f(x)-\frac{1}{m}, & \text{for } f(x)\geq\frac{1}{m} \\
    0, & \text{for } f(x) < \frac{1}{m}
  \end{cases}
  \right. ,
\end{equation}
for $m=1, 2, ...$ . Because $f\in L^1$, thus we can assume that functions $f_m$ have compact support, $m=1, 2, ...$ . By Theorem \ref{Kolyada_zwarte}, there exists a positive constant $\mathcal{C}$ such that 
\begin{equation}\label{g1}
\lVert f_m\rVert_{B_{\Phi,1}^\Psi}
\leq
\mathcal{C}\lVert f_m\rVert_{BV},
\end{equation}
for $m=1, 2, ...$ . By coarea formula we have
\begin{equation}\label{g2}
\begin{aligned}
\lVert\nabla f_m\rVert_M
&=
\int_0^\infty\lVert\nabla\chi_{\{f_m(x)>t\}}\rVert_M\, dt \\
&=
\int_0^\infty\lVert\nabla\chi_{\{f(x)-\frac{1}{m}>t\}}\rVert_M\, dt \\
&=
\int_\frac{1}{m}^\infty\lVert\nabla\chi_{\{f(x)>t\}}\rVert_M\, dt \\
&\leq
\int_0^\infty\lVert\nabla\chi_{\{f(x)>t\}}\rVert_M\, dt \\
&=
\lVert\nabla f\rVert_M.
\end{aligned}
\end{equation}
Also 
\begin{equation}\label{g3}
\lVert f_m\rVert_1
\leq
\lVert f\rVert_1.
\end{equation}
By \eqref{g1},\eqref{g2} and \eqref{g3} we have
$$
\lVert f_m\rVert_{B_{\Phi,1}^\Psi}
\leq
\mathcal{C}\lVert f_m\rVert_{BV}
\leq
\mathcal{C}\lVert f\rVert_{BV}.
$$
Moreover, by Lebesgue's monotone convergence theorem, we have
\begin{eqnarray*}
\lVert f_m\rVert_{B_{\Phi,1}^\Psi}
&=&
\lVert f_m\rVert_\Phi+\int_0^\infty \Psi(t)\omega_\Phi(f_m,t)\frac{dt}{t}
\xrightarrow[m\rightarrow\infty]{}
\lVert f\rVert_\Phi+\int_0^\infty \Psi(t)\omega_\Phi(f,t)\frac{dt}{t} \\
&=&
\lVert f\rVert_{B_{\Phi,1}^\Psi}.
\end{eqnarray*}
Thus, by the above calculations we have
$$
\lVert f\rVert_{B_{\Phi,1}^\Psi}
\leq
\mathcal{C}\lVert f\rVert_{BV}.
$$

If function $f\in BV(\mathbb{R}^d)$ has not only positive values, then we can write it as a difference of two positive functions $f^+$ and $f^-$ with disjoint supports. For these functions, by a previous part of the proof, we have the estimates
$$
\lVert f^+\rVert_{B_{\Phi,1}^\Psi}
\leq
\mathcal{C}\lVert f^+\rVert_{BV}
$$
and
$$
\lVert f^-\rVert_{B_{\Phi,1}^\Psi}
\leq
\mathcal{C}\lVert f^-\rVert_{BV}.
$$
So we can write
$$
\lVert f\rVert_{B_{\Phi,1}^\Psi}
\leq
\lVert f^+\rVert_{B_{\Phi,1}^\Psi}+\lVert f^-\rVert_{B_{\Phi,1}^\Psi}
\leq
\mathcal{C}\left[ \lVert f^+\rVert_{BV} + \lVert f^-\rVert_{BV} \right]
\leq
2\mathcal{C}\lVert f\rVert_{BV}.
$$
This ends the proof.

\section{A concrete example (of functions $\Phi$ and $\Psi$)}
In this section we will construct functions $\Phi$ and $\Psi$ which satisfy condition \eqref{Kolyada_nier_zw} of Theorem \ref{Kolyada_zwarte} for $d=2$. Let
$$\Phi^{-1}(t)=
\begin{cases} 
t\cdot e^{\alpha\frac{\ln(1/\sqrt{t})}{\ln\ln(1/\sqrt{t})}}
&\text{ for } t\in[0,1/r), \\
pt+q
&\text{ for } t\in[1/r,r), \\
t\cdot e^{-\alpha\frac{\ln\sqrt{t}}{\ln\ln\sqrt{t}}}
&\text{ for } t\geq r,
\end{cases}$$
and we put
\begin{eqnarray}
\label{def_psi}
\Psi(t)=\frac{t}{\Phi^{-1}(t^2)}
\end{eqnarray}
for $t\geq 0$.

We define 
$$r=e^{2e^2}$$
and we want from $\alpha$ to be smaller than $e^{-2}$. Then
$$\alpha\frac{\ln r}{\ln\ln r}\leq 1.$$
Moreover we choose $p$ and $q$ in such a way that the function $\Phi^{-1}$ is continuous. So, it is not hard to see that 
\begin{eqnarray*}
p=\frac{re^{-\alpha\frac{\ln(\sqrt{r})}{\ln\ln(\sqrt{r})}}-\frac{1}{r}e^{\alpha\frac{\ln(\sqrt{r})}{\ln\ln(\sqrt{r})}}}{r-1/r}
\mbox{ and }
q=re^{-\alpha\frac{\ln(\sqrt{r})}{\ln\ln(\sqrt{r})}}-pr.
\end{eqnarray*}


Although we define the function $\Phi^{-1}$ for all positive values, we will only consider sufficiently big arguments of this function. The reason is that Theorem \ref{Kolyada_zwarte} deals with the case of a compact domain. This means that one only needs to look at big values of functions to check if the function belongs to the underlying proper Orlicz space. We do not need to be bathered by small values of functions and their rate of convergence at infinity. 

We can summarise the above by saying that from our compact point of view the big values of Young functions are important.

Now, keeping in mind assumption \eqref{def_psi}, we will check the integral condition \eqref{Kolyada_nier_zw} from Theorem \ref{Kolyada_zwarte}, that is we will show the existence of the constant $D>0$ such that for every $s>0$ we have
$$\frac{s}{\Phi^{-1}(s^2)}\int_r^s\frac{1}{t^2\Phi^{-1}\left(\frac{1}{t^2}\right)}\,dt +\int_s^\infty\frac{s}{t^2\Phi^{-1}(ts)\Phi^{-1}\left(\frac{1}{t^2}\right)}\,dt<D.$$
In the compact case the left integral above is over some interval separated from zero. We choose $r$ as a lower integral limit arbitrarily. It can be chosen to be any other strictly positive number. 

Because $r\leq s$, using assumption \eqref{def_psi}, we get for the first integral
\begin{eqnarray*}
\frac{s}{\Phi^{-1}(s^2)}\int_r^s\frac{1}{t^2\Phi^{-1}\left(\frac{1}{t^2}\right)}\, dt
&=&
\frac{s}{s^2e^{-\frac{\alpha\ln{s}}{\ln{\ln{s}}}}}\int_r^s\frac{1}{e^\frac{\alpha\ln{t}}{\ln{\ln{t}}}}\,dt \\
&\leq&
\frac{s}{s^2e^{-\frac{\alpha\ln{s}}{\ln{\ln{s}}}}}\int_r^s\frac{1}{e^\frac{\alpha\ln{t}}{\ln{\ln{s}}}}\,dt \\
&=&
\frac{s}{s^2e^{-\ln{s^\beta}}}\int_r^s\frac{1}{e^{\ln{t^\beta}}}\,dt \\
&=&
s^{\beta-1} \int_r^s t^{-\beta}\,dt \\
&\leq&
\frac{1}{1-\beta} \\
&<&
2.
\end{eqnarray*}
Above we put $\beta=\frac{\alpha}{\ln\ln{s}}$ and used in the last line an appropriate estimate.

For the second integral we will use the substitutions $\ln t=x$ and $\ln r=k$. We then get
\begin{eqnarray*}
\int_s^\infty\frac{s}{t^2\Phi^{-1}(ts)\Phi^{-1}\left(\frac{1}{t^2}\right)}\,dt
&\leq&
\int_r^\infty\frac{s}{t^2\Phi^{-1}(ts)\Phi^{-1}\left(\frac{1}{t^2}\right)}\,dt \\
&=&
\int_r^\infty\frac{1}{t}e^{\alpha\big(\frac{\ln\sqrt{st}}{\ln\ln\sqrt{st}}-\frac{\ln t}{\ln\ln t}\big)}\,dt \\ 
&\leq&
\int_k^\infty e^{\alpha\big(\frac{1/2(x+k)}{\ln(x+k)-\ln 2}-\frac{x}{\ln x}\big)}\,dx \\
&\leq&
\int_k^\infty e^{\alpha\cdot\frac{x\ln x+k\ln x-2x\ln x-2x\ln 2}{2(\ln x-\ln 2)\ln x}}\,dx \\
&=&
\int_k^\infty e^{\alpha\cdot\frac{(k-x)\ln x-2x\ln 2}{2(\ln x-\ln 2)\ln x}}\,dx \\
&\leq&
\int_k^\infty e^{\alpha\cdot\frac{-2x\ln 2}{2(\ln x-\ln 2)\ln x}}\,dx \\
&<&
\infty.
\end{eqnarray*}

\end{document}